\documentclass[12pt]{article}

\usepackage{amsfonts}
\usepackage{amsthm}
\usepackage{euscript}
\usepackage{graphicx}
\usepackage{amsfonts}
\usepackage{amsthm}
\usepackage{euscript}
\usepackage{amsmath}
\usepackage{amssymb}
\usepackage{wrapfig}
\usepackage{caption}
\usepackage{indentfirst}
\usepackage{cmap}
\usepackage{makeidx}

\usepackage[colorlinks, citecolor=black, urlcolor=blue, linkcolor=black]{hyperref}

\newtheorem{theorem}{Theorem}[section]
\newtheorem{lemma}{Lemma}[section]
\newtheorem{corollary}{Corollary}[section]
\newtheorem*{corollary*}{Corollary}

\newtheorem{remark}{Remark}[section]

\makeatletter
\@addtoreset{equation}{section}
\makeatother

\begin{document}

\thispagestyle{plain} 

\sloppy

\begin{center}
{\bf \Large\bf Spectral density in a Moszynski's class of Jacobi matrices. Spectral phase transition of 2nd type.}
\\ \bigskip
{\large Ianovich E.A.}
\\ \bigskip
   {\it Saint Petersburg, Russia}
\\ \bigskip
   {\it\small E-mail: eduard@yanovich.spb.ru}
\end{center}

\begin{abstract}
In this paper it is considered a spectral density for a class of Jacobi matrices with absolutely continuous spectrum that was examined first by Moszynski. It is shown that the corresponding spectral density is equivalent to the positive continuous function everywhere except the point $x=0$. In the point \(x=0\) the spectral density may be finite as well as infinite depending on the parameter. So in this class we discover an example of 2nd type spectral phase transition.
\end{abstract}

\section{Introduction}

Let us consider Moszynski's class of Jacobi matrices
$$
  \left(\begin{array}{ccccc} 0&b_0&0&0&\ldots\\
                      b_0&0&b_1&0&\ldots\\
                      0&b_1&0&b_2&\ldots\\
                      0&0&b_2&0&\ldots\\
                      \vdots&\vdots&\vdots&\vdots&\ddots
  \end{array}\right),
$$
where $b_0>0$, $b_{2k-1}=b_{2k}=k^\alpha\,,\:k=1,2,3,\ldots\:$; $\alpha\in(1/2;1)$. For given $\alpha$ this matrix describes unbounded self-adjoint operator $A$ with absolutely continuous spectrum covering the whole real axis: $\sigma_{ac}(A)=\mathbb{R}$~\cite{1}.
Standard three-term recurrence relations for this matrix have the form
\begin{equation}
\label{recurrent_relations}
b_{n-1}\,y_{n-1}+b_n\,y_{n+1}=x\,y_n\,,\quad (n=1,2,3,\ldots)
\end{equation}
and define so called generalized eigenvectors $\{y_n\}_{n=0}^\infty$. In particular, 1st type orthogonal polynomials $P_n(x)$ are solutions of~(\ref{recurrent_relations}) with initial conditions $P_0(x)=1$, $P_1(x)=x/b_0$. The relations~(\ref{recurrent_relations}) can be written down in the matrix form
$$
v_{n+1}=B_n\,v_n\,,\quad n=1,2,\ldots\,,
$$
where
$$
v_n=\left(\begin{array}{c}
                     y_{n-1} \\
                     y_n \\
                     \end{array}\right)\,,\quad
B_n=\left(
        \begin{array}{cc}
              0 & 1 \\
              -\frac{b_{n-1}}{b_n} & \frac{x}{b_n} \\
            \end{array}
          \right).
$$
In works~\cite{2,3,4,5,6} it was shown that if there exists the limit
\begin{equation}
\label{approximation}
\lim\limits_{n\to\infty}\left(b_nP^2_n(x)-b_{n-1}P_{n-1}(x)P_{n+1}(x)\right)=\lim\limits_{n\to\infty}\Delta_n(x)=\Delta(x)\,,
\end{equation}
then the corresponding spectral density of the operator $A$ equals to
$$
f(x)=\frac{1}{\pi\Delta(x)}.
$$
In particular, if the weights $b_n$ satisfy the conditions
\begin{enumerate}
\label{conditions}
\item $\displaystyle \lim b_n=+\infty$,

\item $\displaystyle \lim\frac{b_n}{b_{n+1}}=1$,

\item $\displaystyle\left\{\frac{b_{n-1}}{b_n}-\frac{b_{n-2}}{b_{n-1}}\right\}
\in l_1$,

\item $\displaystyle\left\{\frac{1}{b_n}-\frac{1}{b_{n-1}}\right\}\in l_1$\,,
\end{enumerate}
the limit function $\Delta(x)$ exists and is positive everywhere so that $f(x)\in C(-\infty,+\infty)$. In our case the third condition failures. We will show that it results to the possibility the density to have an essential singularity at the point \(x=0\). It confirms the hypothesis made in~\cite{11}. Namely, we will show in Theorem~(\ref{main_theorem}) that
$$
f(x)\sim f_0\,|x|^{\textstyle\frac{3\alpha-2}{1-\alpha}}\,,\:x\to 0\,,\quad f_0\ne0\,.
$$
So that
$$
\lim\limits_{x\to0}f(x)=\left\{\begin{array}{l} +\infty\,,\qquad \frac{1}{2}<\alpha<\frac{2}{3}\\
											    f_0\ne0\,,\quad \alpha=\frac{2}{3}\\
											    \:0\,, \quad\qquad \frac{2}{3}<\alpha<1
											    \end{array}\right.\,.
$$

This effect may be called spectral phase transition of 2nd type by analogy in physics. Spectral phase transition of 1st type was considered first in~\cite{12}. It consists of a sudden change in the type of spectrum. Here there is no a sudden change of spectrum type but the spectral density suddenly change their properties. The same situation takes place in physics with 2nd type phase transition~\cite{13}. It explains why we choose such name for this effect.

\section{Asymptotics of generalized eigenvectors and spectral density at \(x\ne 0\)}

To obtain the asymptotics of orthogonal polynomials we will use the following discrete analogue of Levinson-type theorem~\cite{7}:
\begin{theorem}
\label{Levinson}
Let
$$
u_{n+1}=(A_n+V_n)\,u_n\,,\quad n=1,2,\ldots
$$
be a system of recurrence relations, where $A_n,V_n$ are $2\times2$ matrices and $u_n=(u_n^{(1)},u_n^{(2)})^{\top}$. Suppose the matrices $A_n$ are diagonalizable and not degenerate for all $n$ sufficiently large: $A_n=S_n\Lambda_n S_n^{-1}$, $\Lambda_n=diag(\lambda_1(n);\lambda_2(n))$. Let the eigenvalues of $A_n$ satisfy for all sufficiently large $n$ the condition: $|\lambda_1(n)|\le|\lambda_2(n)|$ $($or $|\lambda_1(n)|\ge|\lambda_2(n)|$$)$. Let the norms $\|S_n\|$ be bounded and
$$
S_{n+1}^{-1}S_n=E+R_n\,,\quad \sum_{n=p}^\infty\frac{\|R_n\Lambda_n\|}{|\lambda_i(n)|}<+\infty\,,\quad
\sum_{n=p}^\infty\frac{\|S_{n+1}^{-1}V_nS_n\|}{|\lambda_i(n)|}<+\infty\,,\,i=1,2\,,
$$
where $p$ is large enough, $E$ is identity matrix. Then there exist solutions of recurrence relations satisfying the following asymptotic formulas
$$
u^{(i)}_n=\left(\prod_{j=p}^{n-1}\lambda_i(j)\right)(e_n^{(i)}+\circ(1))\,,\quad n\to\infty\,,\,i=1,2\,,
$$
where $e_n^{(i)}$ are eigenvectors of $A_n$ $($$A_ne_n^{(i)}=\lambda_i(n)e_n^{(i)}$$)$.
\end{theorem}

\begin{remark}
\label{remark1}
If the matrices $A_n$ and $V_n$ depend on the parameter $x\in[a,b\,]$, then the uniform in $x$ fulfilling of all conditions of the theorem means that the term $\circ(1)$ in asymptotic formulas also uniformly tends to zero as $n\to\infty$.
\end{remark}

Let $V_k=0$, $u_k=v_{2k}$ and
$$
A_k=B_{2k+1}B_{2k}=\left(\begin{array}{cc} -\frac{b_{2k-1}}{b_{2k}} & \frac{x}{b_{2k}}\\
                                   -\frac{b_{2k-1}}{b_{2k}}\frac{x}{b_{2k+1}} & \frac{x^2}{b_{2k+1}b_{2k}}-\frac{b_{2k}}{b_{2k+1}}
                \end{array}\right)=
                \left(\begin{array}{cc} -1 & \frac{x}{k^\alpha}\\
                                   -\frac{x}{(k+1)^\alpha} & \frac{x^2}{(k+1)^\alpha k^\alpha}-\frac{k^\alpha}{(k+1)^\alpha}
                \end{array}\right).
$$
Then
$$
u_{k+1}=A_k\,u_{k}.
$$
It is easy to check that if $x\in\mathbb{R}\setminus\{0\}$, then $\mbox{discr}\,A_k<0$ for $k>N$. Hence for all $k$ sufficiently large the matrices $A_k$ have two complex conjugate eigenvalues and it is easy to check that
$$
A_k=S_k\Lambda_kS_k^{-1}\,,\quad (k>N\,,\,x\ne0)\,,
$$
where
$$
\Lambda_k=\left(
            \begin{array}{cc}
              \lambda_k^+ & 0 \\
              0 & \lambda_k^- \\
            \end{array}
          \right)\,,\quad
          \lambda_k^{\pm}=\frac{1}{2}\left(\omega_k\pm i\sqrt{4\frac{k^\alpha}{(k+1)^\alpha}-\omega_k^2}\,\right)\,,\:
\omega_k=\frac{x^2}{((k+1)k)^\alpha}-\frac{k^\alpha}{(k+1)^\alpha}-1\,,
$$
$$
S_k=\left(
            \begin{array}{cc}
              1 & 1 \\
              \frac{k^\alpha(\lambda_k^++1)}{x} & \frac{k^\alpha(\lambda_k^-+1)}{x} \\
            \end{array}
          \right)\,,
$$
so that
$$
(\lambda_k^{\pm}+1)\sim\frac{\alpha}{2k}+\frac{x^2}{2k^{2\alpha}}\pm i\frac{|x|}{k^\alpha}\,,\quad |\lambda_k^{\pm}|=\sqrt{\frac{k^\alpha}{(k+1)^\alpha}}\to 1\,,\,k\to\infty,
$$
$$
S_k\to\left(
            \begin{array}{cc}
              1 & 1 \\
              i\frac{|x|}{x} & -i\frac{|x|}{x} \\
            \end{array}
          \right)\,,\,k\to\infty.
$$
Therefore the norms $\|S_k\|$ are bounded. Next, we have
$$
S_{k+1}^{-1}S_k=\frac{1}{k^\alpha(\lambda_k^--\lambda_k^+)}\cdot
$$
$$
        \cdot\left(
            \begin{array}{cc}
              (k+1)^\alpha(\lambda_{k+1}^-+1)-k^\alpha(\lambda_{k}^++1) & (k+1)^\alpha(\lambda_{k+1}^-+1)-k^\alpha(\lambda_{k}^-+1) \\
              k^\alpha(\lambda_{k}^++1)-(k+1)^\alpha(\lambda_{k+1}^++1) &
              k^\alpha(\lambda_{k}^-+1)-(k+1)^\alpha(\lambda_{k+1}^++1) \\
            \end{array}
          \right)=E+R_k\,,
$$
where
$$
R_k\sim\frac{\frac{\alpha}{2}(\frac{1}{(k+1)^{1-\alpha}}-\frac{1}{k^{1-\alpha}})+
              \frac{x^2}{2}(\frac{1}{(k+1)^{\alpha}}-\frac{1}{k^{\alpha}})}{-2i|x|}
        \left(
            \begin{array}{cc}
              1 & 1\\
              -1 & -1
            \end{array}
          \right).
$$
It follows from this formula that $\{\|R_k\|\}\in l_1$.

Thus all conditions of the Theorem~(\ref{Levinson}) are fulfilled, and there exist solutions $u^{\pm}_k=v^{\pm}_{2k}$ such that
$$
u^{\pm}_k=v^{\pm}_{2k}=\left(\prod_{j=p}^{k-1}\lambda^{\pm}_j\right)(e_k^{\pm}+\circ(1))\,,\quad k\to\infty\,,
$$
where
$$
e_k^{\pm}=\left(
                  \begin{array}{c}
                    1 \\
                    \frac{k^\alpha(\lambda_k^{\pm}+1)}{x} \\
                  \end{array}
                \right)\to\left(
                            \begin{array}{c}
                              1 \\
                              \pm i\frac{|x|}{x} \\
                            \end{array}
                          \right)=e^{\pm}\,,\quad k\to\infty.
$$
Let $\lambda_k^{\pm}=|\lambda_k^{\pm}|\,e^{i\phi^{\pm}_k}=\frac{k^{\alpha/2}}{(k+1)^{\alpha/2}}\,e^{i\phi^{\pm}_k}$, where
$\phi^{\pm}_k\sim\pi\mp\frac{|x|}{k^\alpha}$ as $k\to\infty$. Then
$$
v^{\pm}_{2k}=\frac{p^{\alpha/2}}{k^{\alpha/2}}\,\exp\left(i\,\sum\limits_{j=p}^{k-1}\phi^{\pm}_j\right)(e^{\pm}+\circ(1))\,,\quad k\to\infty.
$$
Using Euler's summation formula, one has
$$
\sum\limits_{j=p}^{k-1}\phi^{\pm}_j=\pi k\mp\frac{|x|\,k^{1-\alpha}}{1-\alpha}\mp c+\circ(1)\,,\quad k\to\infty,
$$
where $c$ is some constant. Hence
$$
v^{\pm}_{2k}=\frac{p^{\alpha/2}}{k^{\alpha/2}}\,\exp\left[i\left(\pi k\mp\frac{|x|\,k^{1-\alpha}}{1-\alpha}\mp c\right)\right]
(e^{\pm}+\circ(1))\,,\quad k\to\infty.
$$
It follows that
$$
y_{2k-1}^{\pm}=\frac{p^{\alpha/2}}{k^{\alpha/2}}\,\exp\left[i\left(\pi k\mp\frac{|x|\,k^{1-\alpha}}{1-\alpha}\mp c\right)\right](1+\circ(1))\,,\quad k\to\infty,
$$
$$
y_{2k}^{\pm}=\pm\,\mbox{sign}(x)\,i\,\frac{p^{\alpha/2}}{k^{\alpha/2}}\,\exp\left[i\left(\pi k\mp
\frac{|x|\,k^{1-\alpha}}{1-\alpha}\mp c\right)\right]
(1+\circ(1))\,,\quad k\to\infty.
$$

It is evident that these asymptotics belong to the two linearly independent solutions of~(\ref{recurrent_relations}). Hence 1st type polynomials $P_n(x)$ are their linear combination with constant coefficients and have the following asymptotic behavior
\begin{equation}
\label{asymptotics_P_n(x)}
\begin{array}{c}
\displaystyle P_{2k-1}(x)=(-1)^k\frac{C}{k^{\alpha/2}}\,\cos\left(\frac{|x|\,k^{1-\alpha}}{1-\alpha}+\beta\right)+\circ\left(\frac{1}{k^{\alpha/2}}\right)\,,\quad k\to\infty, \\
\displaystyle
P_{2k}(x)=(-1)^k\,\mbox{sign}(x)\,\frac{C}{k^{\alpha/2}}\,\sin\left(\frac{|x|\,k^{1-\alpha}}{1-\alpha}+\beta\right)+\circ\left(\frac{1}{k^{\alpha/2}}\right)\,,
\quad k\to\infty,
\end{array}
\end{equation}
where $C>0$ and $\beta\in\mathbb{R}$ are some constants (depending on $x$ generally). This formulas were obtained under the condition $x\ne 0$. From remark~(\ref{remark1}) it follows that the term $\circ(1)$ tends to zero uniformly on any set of the form $[-A;-B]\cup[B;A]$ ($0<B<A$). If $x=0$ then the matrices $A_k$ become diagonal and one can easily obtain that
\begin{equation}
\label{asymptotics_P_n(0)}
\begin{array}{c}
\displaystyle P_{2k-1}(0)=0\,, \\
\displaystyle P_{2k}(0)=(-1)^k\,\frac{b_0}{k^\alpha}.
\end{array}
\end{equation}
It follows that the asymptotics at $x=0$ is changed. It reflects the fact of appearance of eigenvalue at $x=0$ when $\alpha>1/2$~\cite{1,5}.

Substituting~(\ref{asymptotics_P_n(x)}) and~(\ref{asymptotics_P_n(0)}) in~(\ref{approximation}), one obtains
\begin{equation}
\label{limit_Delta_n(x)}
\Delta(x)=\lim\limits_{n\to\infty}\left(b_nP^2_n(x)-b_{n-1}P_{n-1}(x)P_{n+1}(x)\right)=C^2\equiv C^2(x)>0\,,\:x\ne 0\,,
\end{equation}
$$
\Delta(0)=0\,,
$$
so that
\begin{equation}
\label{formula_for_spectral_density}
f(x)=\frac{1}{\pi C^2(x)}\,,\:x\ne 0.
\end{equation}

Due to uniform vanishing of $\circ(1)$ in asymptotic formulas~(\ref{asymptotics_P_n(x)}), the limit~(\ref{limit_Delta_n(x)})
is also uniform on any set of the form $[-A;-B]\cup[B;A]$ ($0<B<A$). Hence the function $\Delta(x)=C^2(x)$ is continuous everywhere except maybe the point $x=0$ ! Thus we obtain

\begin{theorem}
Spectral density of the considered class of operators is equivalent to the positive continuous function everywhere except maybe the point $x=0$.
\end{theorem}

\section{Investigation of spectral density at the point \(x=0\)}

Let's present the matrix \(A_k\) in the form
$$
A_k=B_{2k+1}B_{2k}=
                \left(\begin{array}{cc} -1 & \frac{x}{k^\alpha}\\
                                   -\frac{x}{(k+1)^\alpha} & \frac{x^2}{(k+1)^\alpha k^\alpha}-\frac{k^\alpha}{(k+1)^\alpha}
                \end{array}\right)=
$$
$$
=\left(\begin{array}{cc} -1 & 0\\
                                   0 & -\frac{k^\alpha}{(k+1)^\alpha}
                \end{array}\right)+\left(\begin{array}{cc} 0 & \frac{x}{k^\alpha}\\
                                   -\frac{x}{(k+1)^\alpha} & 0
                \end{array}\right)+\left(\begin{array}{cc} 0 & 0\\
                                   0 & \frac{x^2}{(k+1)^\alpha k^\alpha}
                \end{array}\right)=
$$
$$
=-\frac{k^{\alpha/2}}{(k+1)^{\alpha/2}}\left(\begin{array}{cc} \frac{(k+1)^{\alpha/2}}{k^{\alpha/2}} & 0\\
                                   0 & \frac{k^{\alpha/2}}{(k+1)^{\alpha/2}}
                \end{array}\right)+\frac{x}{k^\alpha}\left(\begin{array}{cc} 0 & 1\\
                                   -1 & 0
                \end{array}\right)+
                \left(\frac{x}{k^\alpha}-\frac{x}{(k+1)^\alpha}\right)		 
               	\left(\begin{array}{cc} 0 & 0\\
                                   1 & 0
                \end{array}\right)+
$$
$$
+\frac{x^2}{(k+1)^\alpha k^\alpha}	\left(\begin{array}{cc} 0 & 0\\
                                   0 & 1
                \end{array}\right)=-\frac{k^{\alpha/2}}{(k+1)^{\alpha/2}}\left(E+\frac{\alpha}{2k}
                \left(\begin{array}{cc} 1 & 0\\
                                   0 & -1
                \end{array}\right)+O\left(\frac{1}{k^2}\right)\right)+\frac{x}{k^\alpha}
                \left(\begin{array}{cc} 0 & 1\\
                                       -1 & 0
                \end{array}\right)+
$$
$$
+\left(\frac{x}{k^\alpha}-\frac{x}{(k+1)^\alpha}\right)		 
               	\left(\begin{array}{cc} 0 & 0\\
                                   1 & 0
                \end{array}\right)+\frac{x^2}{(k+1)^\alpha k^\alpha}	
                \left(\begin{array}{cc} 0 & 0\\
	                                   0 & 1
                \end{array}\right)\,.
$$
Putting out of brackets the factor \(-\dfrac{k^{\alpha/2}}{(k+1)^{\alpha/2}}\), we find
$$
A_k=-\frac{k^{\alpha/2}}{(k+1)^{\alpha/2}}\left(E-\frac{x}{k^\alpha}
                \left(\begin{array}{cc} 0 & 1\\
                                       -1 & 0
                \end{array}\right)-\frac{x^2}{(k+1)^\alpha k^\alpha}	
                \left(\begin{array}{cc} 0 & 0\\
	                                   0 & 1
                \end{array}\right)+\frac{\alpha}{2k}
                \left(\begin{array}{cc} 1 & 0\\
                                   0 & -1
                \end{array}\right)+R_k\right)\,,
$$
where \(\{\|R_k\|\}\in l_1\) uniformly on any bounded interval of \(x\). Suppose for the simplicity that \(\alpha\in(1/2,1)\). Then we can write
$$
A_k=-\frac{k^{\alpha/2}}{(k+1)^{\alpha/2}}\left(E-\frac{x}{k^\alpha}
                \left(\begin{array}{cc} 0 & 1\\
                                       -1 & 0
                \end{array}\right)+\frac{\alpha}{2k}
                \left(\begin{array}{cc} 1 & 0\\
                                   0 & -1
                \end{array}\right)+R_k\right)\,,
$$
where \(R_k\) have the same properties as before (we will save the same notation). Introducing notations
\begin{equation}
\label{notations_P_S}
S=\left(\begin{array}{cc} 0 & 1\\
                                   -1 & 0
                \end{array}\right)\,,\quad 
P=\left(\begin{array}{cc} 1 & 0\\
                                   0 & -1
                \end{array}\right)\,,\quad S^2=-E\,,\quad P^2=E\,,\quad 
PS=-SP\,,
\end{equation}
and remembering that \(\alpha\in(1/2,1)\), we can rewrite previous expression as
$$
A_k=-\frac{k^{\alpha/2}}{(k+1)^{\alpha/2}}\left(E-\frac{x}{k^\alpha}\,S+\frac{\alpha}{2k}\,P+R_k\right)=-\frac{k^{\alpha/2}}{(k+1)^{\alpha/2}}\left(\exp\left(-\frac{x}{k^\alpha}\,S\right)+\frac{\alpha}{2k}\,P+R'_k\right)\,,
$$
where \(R'_k\) has the same properties as \(R_k\). We have next
$$
A_k=-\frac{k^{\alpha/2}}{(k+1)^{\alpha/2}}\exp\left(-\frac{x}{k^\alpha}\,S\right)\left(E+\frac{\alpha}{2k}\exp\left(\frac{x}{k^\alpha}\,S\right)P+R''_k\right)\,,
$$
or
\begin{equation}
\label{formula_1}
A_k=-\frac{k^{\alpha/2}}{(k+1)^{\alpha/2}}\,U_k\left(E+T_k+R''_k\right)\,,
\end{equation}
where
$$
U_k=\exp\left(-\frac{x}{k^\alpha}\,S\right)\,,\quad T_k=\frac{\alpha}{2k}\exp\left(\frac{x}{k^\alpha}\,S\right)P\,,\quad R''_k=\exp\left(\dfrac{x}{k^\alpha}\,S\right)R'_k\,.
$$

From~(\ref{notations_P_S}) it follows that for \(t\in\mathbb{R}\) the matrix \(\exp\left(tS\right)\)
is orthogonal: \(U_kU_k^T=U_k^TU_k=E\), so that \(\|R''_k\|=\|R'_k\|\) (analogous idea used in~\cite{8}). Besides,
\begin{equation}
\label{matrix_representation_exp(S)}
e^{t\,S}=\left(\begin{array}{cc} \cos t & \sin t\\
                                   -\sin t & \cos t
                \end{array}\right)\,.
\end{equation}
Using~(\ref{matrix_representation_exp(S)}), we obtain 
\begin{equation}
\label{asymptotics_F_k}
F_k\equiv U_kU_{k-1}\ldots U_1=\exp\left(-x\sum\limits_{n=1}^k\frac{1}{n^\alpha}\,S\right)=
\left(\begin{array}{cc} \cos\phi_k & \sin\phi_k\\
                                   -\sin\phi_k & \cos\phi_k
                \end{array}\right)(E+o(1))\,,
\end{equation}
where \(\phi_k=-\dfrac{x\,k^{1-\alpha}}{1-\alpha}+xc\) (\(c\) is constant). So, the product \(U_kU_{k-1}\ldots U_1\) insures oscillating behaviour of asymptotics~(\ref{asymptotics_P_n(x)}).
Taking it into account, we have from~(\ref{formula_1}) (omitting strokes at \(R''_k\))
$$
A_kA_{k-1}\ldots A_1=\frac{(-1)^k}{(k+1)^{\alpha/2}}\,U_k(E+T_k+R_k)U_{k-1}(E+T_{k-1}+R_{k-1})\ldots
U_1T_1(E+T_1+R_1)=
$$
$$
=\frac{(-1)^k}{(k+1)^{\alpha/2}}\,U_kU_{k-1}(E+U_{k-1}^TT_kU_{k-1}+U_{k-1}^TR_kU_{k-1})(E+T_{k-1}+R_{k-1})\ldots U_1T_1(E+R_1)\,.
$$
Here we used again orthogonality of \(U_k\). Due to this orthogonality we have
\(\|U_{k-1}^TR_kU_{k-1}\|=\|R_k\|\) (\(U_k^*=U_k^T\), since the matrix \(U_k\) is real).
Continuing this process, we obtain 
$$
A_kA_{k-1}\ldots A_1=\frac{(-1)^k}{(k+1)^{\alpha/2}}\,F_k(E+T'_k+R'_k)(E+T'_{k-1}+R'_{k-1})\ldots (E+T'_1+R'_1)\,,
$$
where
$$
T'_k=F_{k-1}^TT_kF_{k-1}\,,\quad R'_k=F_{k-1}^TR_kF_{k-1}\,,\:\|R'_k\|=\|R_k\|\,,\:k\ge2\,,\quad
T'_1=T_1\,,\:R'_1=R_1\,. 
$$
Using~(\ref{matrix_representation_exp(S)}), we have
$$
T'_k=F_{k-1}^TT_kF_{k-1}=\frac{\alpha}{2k}\,F_{k-1}^T\,U_k^TPF_{k-1}=\beta_k\,F_k^TPF_{k-1}=
$$
$$
=\beta_k\left(\begin{array}{cc} \cos t_k & -\sin t_k\\
                                   \sin t_k & \cos t_k
                \end{array}\right)        
                \left(\begin{array}{cc} 1 & 0\\
                                   0 & -1
                \end{array}\right)        
      \left(\begin{array}{cc} \cos t_{k-1} & \sin t_{k-1}\\
                                   -\sin t_{k-1} & \cos t_{k-1}
                \end{array}\right)=   
$$
$$                
=\beta_k\left(\begin{array}{cc} 
		\cos(t_k+t_{k-1}) & \sin(t_k+t_{k-1})\\
       \sin(t_k+t_{k-1}) & -\cos(t_k+t_{k-1})
                \end{array}\right)\,,
$$
where
$$
t_k=-x\sum\limits_{n=1}^k\frac{1}{n^\alpha}\,,\quad \beta_k=\frac{\alpha}{2k}\,.
$$
So we obtain
\begin{equation}
\label{formula_prod_A_i}
A_kA_{k-1}\ldots A_1=\frac{(-1)^k}{(k+1)^{\alpha/2}}\,F_kG_k\,,
\end{equation}
where
$$
G_k=(E+T'_k+R'_k)(E+T'_{k-1}+R'_{k-1})\ldots (E+T'_1+R'_1)\equiv\prod_{i=1}^k(E+T'_i+R'_i)\,,
$$
\(\{\|R'_k\|\}\in l_1\) uniformly on any bounded interval of \(x\) and \(F_k\) is defined by~(\ref{asymptotics_F_k}). Since
$$
t_k=-x\sum\limits_{n=1}^k\frac{1}{n^\alpha}=-\dfrac{x\,k^{1-\alpha}}{1-\alpha}+xc+O\left(\frac{x}{k^\alpha}\right)=\phi_k+O\left(\frac{x}{k^\alpha}\right)\,,\quad k\to\infty\,,\quad (c=const)\,,
$$
$$
\phi_k=-\dfrac{x\,k^{1-\alpha}}{1-\alpha}+xc\,,
$$
we can write
$$
T'_k=\frac{\alpha}{2k}\left(\begin{array}{cc} 
		\cos(t_k+t_{k-1}) & \sin(t_k+t_{k-1})\\
       \sin(t_k+t_{k-1}) & -\cos(t_k+t_{k-1})
                \end{array}\right)=
$$
$$                
=\frac{\alpha}{2k}\left(\begin{array}{cc} 
		\cos2\phi_k & \sin2\phi_k\\
       \sin2\phi_k & -\cos2\phi_k
                \end{array}\right)+O\left(\frac{x}{k^{1+\alpha}}\right)\,,\quad k\to\infty\,,
$$
and
$$
G_k=\prod_{n=1}^k(E+H_n+D_n)\,,
$$
where
$$
H_n=\frac{\alpha}{2n}\left(\begin{array}{cc} 
		\cos2\phi_n & \sin2\phi_n\\
       \sin2\phi_n & -\cos2\phi_n
                \end{array}\right)\,,
$$
and \(\{\|D_n\|\}\in l_1\) uniformly on any bounded interval of \(x\).

Let's prove that the matrix product \(G_k\) converges if \(x\ne 0\). The convergence is not absolute, since \(\sum\limits_n\|H_n\|=\infty\). To establish conditionally convergence of \(G_k\), we will use the following theorem (see~\cite{9}, Theorem 3.2) :

\begin{theorem}
\label{infinite_product_convergence}
Let \(\{V_n\}_1^\infty\) be a sequence of square matrices with \(\|V_n\|<1\).
Suppose the matrix series \(\sum\limits_{n=1}^\infty V_n\) conditionally converges, and let 
\(Q_n=\sum\limits_{k=n}^\infty V_k\). If
$$
\sum\limits_{n=1}^\infty\|V_nQ_n\|<\infty\,,
$$
then the matrix product \(\prod\limits_{k=1}^n(E+V_k)\) conditionally converges to an invertible matrix.
\end{theorem}

In our case \(V_n=H_n+D_n\). The convergence of \(\sum\limits_{n=1}^\infty V_n\) is equivalent to the convergence of series 
$$
\sum\limits_{n=1}^\infty\dfrac{\sin2\phi_n}{n}\:\:\mbox{and}\:\:\sum\limits_{n=1}^\infty\dfrac{\cos2\phi_n}{n}
$$ 
or
$$
\sum\limits_{n=1}^\infty\dfrac{\sin(\beta\,n^s+\gamma)}{n}\:\:\mbox{and}\:\:\sum\limits_{n=1}^\infty\dfrac{\cos(\beta\,n^s+\gamma)}{n}\,,
$$ 
where \(\beta=-\dfrac{2x}{1-\alpha}\), \(s=1-\alpha\), \(\gamma=2xc\). Consider for example the second one. Let
$$
f(t)=\dfrac{\cos(\beta\,t^s+\gamma)}{t}\,.
$$
We have the formula~\cite{10} (\(m<n)\) :
\begin{equation}
\label{sum_integral_formula}
\sum_{k=m}^{n}f(k)=\int\limits_m^nf(t)\,dt+\frac{f(n)+f(m)}{2}-\int\limits_m^nP_1(t)f'(t)\,dt\,,
\end{equation}
$$
P_1(t)=[t]-t+\frac{1}{2}=\frac{1}{2}-\{t\}\,,\quad |P_1(t)|\le\frac{1}{2}\,.
$$
From this formula it follows that the convergence of series 
$$
\sum\limits_{n=1}^\infty\dfrac{\cos(\beta\,n^s+\gamma)}{n}\quad (\beta\ne 0)
$$
is equivalent to the convergence of integrals
$$
\int\limits_1^{+\infty}f(t)\,dt\,,\quad \int\limits_1^{+\infty}|f'(t)|\,dt\,.
$$
First of these integrals can be transformed to
\begin{equation}
\label{formula_2}
\int\limits_1^{+\infty}f(t)\,dt=\int\limits_1^{+\infty}\dfrac{\cos(\beta\,t^s+\gamma)}{t}\,dt=
(y=\beta\,t^s)=\frac{1}{s}\int\limits_\beta^{+\infty}\dfrac{\cos(y+\gamma)}{y}\,dy\,.
\end{equation}
(We can consider \(x<0\), so \(\beta>0\).) The last integral converges by Dirichlet's test. The convergence of second integral
$$
\int\limits_1^{+\infty}|f'(t)|\,dt
$$
can be easily proved by calculating the derivative \(f'(t)\).

By the same way one can prove the convergence of series
$$
\sum\limits_{n=1}^\infty\dfrac{\sin(\beta\,n^s+\gamma)}{n}\quad (\beta\ne 0)\,.
$$
Using again the formula~(\ref{sum_integral_formula}), we obtain
$$
\sum\limits_{k=n}^\infty\frac{\cos(\beta\,k^s+\gamma)}{k}=\sum_{k=n}^{\infty}f(k)=\int\limits_n^{+\infty}f(t)\,dt+\frac{f(n)}{2}-\int\limits_n^{+\infty}P_1(t)f'(t)\,dt\,.
$$
It is easy to check that
$$
\int\limits_n^{+\infty}P_1(t)f'(t)\,dt=O\left(\frac{1}{n^{2-s}}\right)=O\left(\frac{1}{n^{1+\alpha}}\right)\,,\quad n\to\infty\,.
$$
Using the same substitution as in~(\ref{formula_2}), we have
$$
\int\limits_n^{+\infty}f(t)\,dt=\frac{1}{s}\int\limits_{\beta\,n^s}^{+\infty}\dfrac{\cos(y+\gamma)}{y}\,dy\,.
$$
Integrating by parts, we find
\begin{equation}
\label{asymptotics_of_int_cos}
\int\limits_n^{+\infty}f(t)\,dt=O\left(\frac{1}{n^s}\right)=O\left(\frac{1}{n^{1-\alpha}}\right)\,,\quad n\to\infty\,,
\end{equation}
so that
$$
\sum_{k=n}^{\infty}f(k)=O\left(\frac{1}{n^{1-\alpha}}\right)\,,\quad n\to\infty\,.
$$
The same estimation will take place and for the \(\sin\)-sum :
$$
\sum\limits_{k=n}^\infty\frac{\sin(\beta\,k^s+\gamma)}{k}=O\left(\frac{1}{n^{1-\alpha}}\right)\,,\quad n\to\infty\,.
$$
It follows that \(\|Q_n\|=O\left(\dfrac{1}{n^{1-\alpha}}\right)\) and \(\|V_nQ_n\|=O\left(\dfrac{1}{n^{2-\alpha}}\right)\), so that \(\{V_nQ_n\}\in l_1\).

Applying Theorem~(\ref{infinite_product_convergence}), we obtain that for \(x\ne 0\) the product \(G_k\) converges to an invertible matrix :
$$
\lim\limits_{k\to\infty}G_k\equiv G(x)\equiv \prod\limits_{n=1}^\infty(E+H_n+D_n)\,.
$$
To understand the form of \(G(x)\), let us consider a continuous analogue of \(G_k\). From definition
$$
G_k=\prod_{n=1}^k(E+H_n+D_n)
$$
we have the following recurrent relations
$$
G_{n+1}=(E+H_{n+1}+D_{n+1})\,G_n\,.
$$
Omitting the uniformly with \(x\) summable perturbation \(D_n\) that not changes essentially the asymptotics~\cite{9}, we have
\begin{equation}
\label{discrete_system}
G_{n+1}=(E+H_{n+1})\,G_n\,,
\end{equation}
or
$$
G_{n+1}-G_n=H_{n+1}\,G_n\,.
$$
Replacing \((G_{n+1}-G_n)\) by \(\dfrac{dG_n}{dn}\), we obtain the following matrix differential equation
$$
\frac{dG_n}{dn}=H_n\,G_n
$$
or
\begin{equation}
\label{equation_G_n}
\frac{dG_n}{dn}=\frac{\alpha}{2n}\left(\begin{array}{cc} 
		\cos(\beta n^s+\gamma) & \sin(\beta n^s+\gamma)\\
       \sin(\beta n^s+\gamma) & -\cos(\beta n^s+\gamma)
                \end{array}\right)G_n\,.
\end{equation}
Let \(Z_n\) be solution of this equation with diagonal matrix:
$$
Z_n=\left(\begin{array}{cc} 
		e^{c_n} & 0\\
        0 & e^{-c_n}
                \end{array}\right)\,,\quad c_n=\frac{\alpha}{2}\int\limits_1^n\frac{\cos(\beta t^s+\gamma)}{t}\,dt\,.
$$
Then after substitution \(G_n=Z_n\,V_n\), we have
$$
\frac{dV_n}{dn}=\frac{\alpha\sin(\beta n^s+\gamma)}{2n}\,Z_n^{-1}\left(\begin{array}{cc} 
		0 & 1 \\
        1 & 0
                \end{array}\right)Z_n\,V_n
$$
or
\begin{equation}
\label{equation_V_n}
\frac{dV_n}{dn}=\frac{\alpha\sin(\beta n^s+\gamma)}{2n}\,W_n\,V_n\,,
\end{equation}
where
$$
W_n=\left(\begin{array}{cc} 
		0 & e^{-2c_n} \\
        e^{2c_n} & 0
                \end{array}\right)\,.
$$
Because 
$$
c_n\to c=\frac{\alpha}{2}\int\limits_1^{+\infty}\frac{\cos(\beta t^s+\gamma)}{t}\,dt\,,
$$
the matrix \(W_n\) tends to the constant matrix $W$ at $n\to\infty$ :
$$
\lim\limits_{n\to\infty}W_n=W=\left(\begin{array}{cc} 
		0 & e^{-2c} \\
        e^{2c} & 0
                \end{array}\right)\,.
$$
So, we can regard that in~(\ref{equation_V_n}) for large $n$ the matrix $W_n$ is constant in compare with  oscillating sin-factor. If we replace in~(\ref{equation_V_n}) \(W_n\) by \(W\), we can solve this equation exactly :
\begin{equation}
\label{equation_V}
\frac{dV_n}{dn}=\frac{\alpha\sin(\beta n^s+\gamma)}{2n}\,W\,V_n\,,
\end{equation}
$$ 
V_n=e^{s_nW}\,V_1\,,\quad  s_n=\frac{\alpha}{2}\int\limits_1^{n}\frac{\sin(\beta t^s+\gamma)}{t}\,dt\,,\quad s_n\to s_{\infty}=\frac{\alpha}{2}\int\limits_1^{+\infty}\frac{\sin(\beta t^s+\gamma)}{t}\,dt\,,
$$
and 
$$
\lim\limits_{n\to\infty} V_n=V=e^{s_\infty W}\,V_1\,.
$$
Thus we have \(G_n=Z_nV_n\to G=Z\,V\), where
$$
Z=\lim\limits_{n\to\infty}Z_n=\left(\begin{array}{cc} 
		e^{c} & 0\\
        0 & e^{-c}
                \end{array}\right)\,,\quad V=e^{s_\infty W}\,V_1\,.
$$
Because \(W^2=E\), expanding the exponent in a Taylor series, we can write
$$
e^{s_\infty W}=\mbox{ch}\,s_\infty E+\mbox{sh}\,s_\infty W
$$ 
and
$$
G=\left(\begin{array}{cc} 
		e^{c}\,\mbox{ch}\,s_\infty & e^{-c}\,\mbox{sh}\,s_\infty \\
        e^{c}\,\mbox{sh}\,s_\infty & e^{-c}\,\mbox{ch}\,s_\infty
                \end{array}\right)\,V_1\,.
$$
The matrix \(V_1\) is determined by initial condition : \(V_1=G_1=E+H_1+D_1\).
  
Our goal now is to find the behaviour of \(G\equiv G(x)\) at \(x\to 0\). The transformation as in formula~(\ref{formula_2}) gives
$$
c\sim-\frac{\alpha}{2s}\ln|\beta|\,,\: x\to 0\,,\qquad
\lim\limits_{x\to 0}s_\infty=s_0\,,\: |s_0|<+\infty\,.
$$
Remember that \(\beta=-\dfrac{2x}{1-\alpha}\), \(s=1-\alpha\), we obtain
$$
e^c\sim\delta\,|x|^{-\dfrac{\alpha}{2(1-\alpha)}}\,,\:x\to 0\,,
$$
where \(\delta\) is non zero constant. Then as \(x\to 0\)
\begin{equation}
\label{G_xto0}
G\sim\left(\begin{array}{cc} 
		\delta\,\mbox{ch}\,s_0\,|x|^{\textstyle-\frac{\alpha}{2(1-\alpha)}} & \delta\,\mbox{sh}\,s_0\,|x|^{\textstyle\frac{\alpha}{2(1-\alpha)}} \\
        \delta\,\mbox{sh}\,s_0\,|x|^{\textstyle-\frac{\alpha}{2(1-\alpha)}} & \delta\,\mbox{ch}\,s_0\,|x|^{\textstyle\frac{\alpha}{2(1-\alpha)}}
                \end{array}\right)\,V_1\,.
\end{equation}
Note that careful analysis of above formulas shows that \(D_1\) is a diagonal matrix at \(x=0\) and non diagonal elements of it are \(O(x)\) as \(x\to 0\). Thus as \(x\to 0\) we have
\begin{equation}
\label{V1_xto0}
V_1=E+H_1+D_1\to \left(\begin{array}{cc} 
		v_{11} & O(x) \\
        O(x) & v_{22}
                \end{array}\right)\,,\quad v_{11}, v_{22}\ne 0\,.
\end{equation}
Taking into account (\ref{asymptotics_F_k}) and (\ref{formula_prod_A_i}), we can write
$$
A_kA_{k-1}\ldots A_1=\frac{(-1)^k}{(k+1)^{\alpha/2}}\left(\begin{array}{cc} \cos\phi_k & \sin\phi_k\\
                                   -\sin\phi_k & \cos\phi_k
                \end{array}\right)G(x)\,(E+o(1))\,,\quad k\to\infty\,,
$$
and since
$$
\left(\begin{array}{c}
                     P_{2k+1} \\
                     P_{2k+2} \\
                     \end{array}\right)=A_kA_{k-1}\ldots A_1\left(\begin{array}{c}
                     P_{1} \\
                     P_{2} \\
                     \end{array}\right)\,,
$$
compare these formulas to the asymptotic formulas (\ref{asymptotics_P_n(x)}) for 1st type polynomials \(P_n(x)\), we find that the function \(C(x)\) in (\ref{asymptotics_P_n(x)}) is defined by
$$
G(x)\left(\begin{array}{c}
                     P_{1} \\
                     P_{2} \\
                     \end{array}\right)\,.
$$
Substituting (\ref{G_xto0}),(\ref{V1_xto0}) and \(P_1(x)=\dfrac{x}{b_0}\), 
\(P_2(x)=\dfrac{x^2-b_0^2}{b_0\,b_1}\) in this expression, we obtain as \(x\to 0\)
$$
G(x)\left(\begin{array}{c}
                     P_{1} \\
                     P_{2} 
                     \end{array}\right)\sim\left(\begin{array}{cc} 
		\delta\,\mbox{ch}\,s_0\,|x|^{\textstyle-\frac{\alpha}{2(1-\alpha)}} & \delta\,\mbox{sh}\,s_0\,|x|^{\textstyle\frac{\alpha}{2(1-\alpha)}} \\
        \delta\,\mbox{sh}\,s_0\,|x|^{\textstyle-\frac{\alpha}{2(1-\alpha)}} & \delta\,\mbox{ch}\,s_0\,|x|^{\textstyle\frac{\alpha}{2(1-\alpha)}}
                \end{array}\right)
                \left(\begin{array}{cc} 
		v_{11} & O(x) \\
        O(x) & v_{22}
                \end{array}\right)
                \left(\begin{array}{c}
                     \dfrac{x}{b_0} \\
                     \dfrac{x^2-b_0^2}{b_0\,b_1} 
                     \end{array}\right)=
$$ 
$$
=|x|^{\textstyle\frac{2-3\alpha}{2(1-\alpha)}}\,\vec{e}\,,\quad \|\vec{e}\|\ne0\,.
$$
Hence
$$
C(x)\sim C_0\,|x|^{\textstyle\frac{2-3\alpha}{2(1-\alpha)}}\,,\:x\to 0\,,\quad C_0\ne0\,.
$$
From (\ref{formula_for_spectral_density}) we obtain the following result:
\begin{theorem}
\label{main_theorem}
If \(\:\frac{1}{2}<\alpha<1\), then the spectral density \(f(x)\) has the following asymptotic behaviour at \(x=0\):
$$
f(x)\sim f_0\,|x|^{\textstyle\frac{3\alpha-2}{1-\alpha}}\,,\:x\to 0\,,\quad f_0\ne0\,.
$$
\end{theorem}
\begin{corollary*}
In this class of operators we have an example of 2nd type spectral phase transition. In particular
$$
\lim\limits_{x\to0}f(x)=\left\{\begin{array}{l} +\infty\,,\qquad \frac{1}{2}<\alpha<\frac{2}{3}\\
											    f_0\ne0\,,\quad \alpha=\frac{2}{3}\\
											    \:0\,, \quad\qquad \frac{2}{3}<\alpha<1
											    \end{array}\right.\,.
$$
\end{corollary*}

Let's consider now the solution \(V_n\) in more detail. The next lemma establishes a connection between solutions of equations~(\ref{equation_V_n}) and~(\ref{equation_V}).

\begin{lemma}
\label{asymptotically_equivalence}
For any solution \(\tilde{V}_n\) of the equations~(\ref{equation_V}) there exist the solution \(V_n\) of the equation~(\ref{equation_V_n}) which is asymptotically equivalent to \(\tilde{V}_n\) :
$$
V_n=\tilde{V}_n(E+o(1))\,,\quad n\to\infty\,.
$$
\end{lemma}
\begin{proof}
We have
$$
\frac{dV_n}{dn}=\frac{\alpha\sin(\beta n^s+\gamma)}{2n}W_nV_n
$$
$$
\frac{d\tilde{V}_n}{dn}=\frac{\alpha\sin(\beta n^s+\gamma)}{2n}W\,\tilde{V}_n
$$
Transforming the first equation
$$
\frac{dV_n}{dn}=\frac{\alpha\sin(\beta n^s+\gamma)}{2n}W_nV_n=
\frac{\alpha\sin(\beta n^s+\gamma)}{2n}W\,V_n+\frac{\alpha\sin(\beta n^s+\gamma)}{2n}(W_n-W)\,V_n\,,
$$
we can write this differential equation as integral one
\begin{equation}
\label{asymptot_V_n}
V_n=\tilde{V}_n\left(E-\frac{\alpha}{2}\int\limits_n^{+\infty}\frac{\sin(\beta t^s+\gamma)}{t}\,\tilde{V}_t^{-1}\,(W_t-W)\,V_t\,dt\right)\,.
\end{equation}
It is easily verified by differentiation. Let's prove that the integral in this expression converges. 
First, the function \(\tilde{V}_t=e^{s_t W}\,{\tilde V}_1\) is obviously bounded. Secondly, we have
$$
W_t-W=\left(\begin{array}{cc} 
		0 & e^{-2c_n}-e^{-2c} \\
        e^{2c_n}-e^{2c} & 0
                \end{array}\right)=\left(\begin{array}{cc} 
		0 & e^{-2c}(e^{2(c-c_n)}-1) \\
        e^{2c}(e^{2(c_n-c)}-1) & 0
                \end{array}\right)\,.
$$
Because (see formula (\ref{asymptotics_of_int_cos}) )
$$
c-c_n=\frac{\alpha}{2}\int\limits_n^{+\infty}\frac{\cos(\beta t^s+\gamma)}{t}\,dt=O\left(\frac{1}{n^{1-\alpha}}\right)\,,\quad n\to\infty\,,
$$
we have \(\|W_t-W\|=O\left(\dfrac{1}{t^{1-\alpha}}\right)\), \(t\to+\infty\). Let's prove that the function \(V_t\) is also bounded. Making Harris-Lutz transform
$$
V_n=(E+Q_n)\,Y_n\,,\quad Q_n=-\frac{\alpha}{2}\int\limits_n^{+\infty}\frac{\sin(\beta t^s+\gamma)}{t}\,W_t\,dt=O\left(\frac{1}{n^{1-\alpha}}\right)\,,\quad n\to\infty\,,
$$ 
we obtain for \(Y_n\) the following equation
$$
\frac{dY_n}{dn}=\frac{\alpha\sin(\beta n^s+\gamma)}{2n}\,(E+Q_n)^{-1}\,Q_nY_n\,,\quad n\ge N\,.
$$
The restriction \(n\ge N\) is motivated by the existence of \((E+Q_n)^{-1}\). Passing to the integral equation
$$
Y_n=Y_N+\frac{\alpha}{2}\int\limits_N^{n}\frac{\sin(\beta t^s+\gamma)}{t}\,(E+Q_t)^{-1}\,Q_t\,Y_t\,dt\,,
$$
we have the estimate
$$
\|Y_n\|\le\|Y_N\|+\int\limits_N^{n}\frac{\alpha}{2\,t}\,\|(E+Q_t)^{-1}\|\,\|Q_t\|\,\|Y_t\|\,dt\,.
$$
Applying Gronwall's inequality, we obtain
$$
\|Y_n\|\le\|Y_N\|\exp\left\{\int\limits_N^{+\infty}\frac{\alpha}{2\,t}\,\|(E+Q_t)^{-1}\|\,\|Q_t\|\,\|Y_t\|\,dt\right\}<+\infty\,.
$$
Thus \(\|V_t\|<C<+\infty\) for \(t\ge 1\) and the integral in~(\ref{asymptot_V_n}) is convergent.
\end{proof}

\begin{corollary}
\label{finite_limit}
For any solution \(V_n\) of the equation (\ref{equation_V_n}) there exists a finite limit \(\lim\limits_{n\to\infty}V_n\).
\end{corollary}

Suppose that \(\tilde{V}_n=e^{s_n W}\), \(\tilde{V}_1=E\). Because \(W^2=E\), expanding the exponent in a Taylor series, we can write
$$
\tilde{V}_n= e^{s_n W}=\mbox{ch}\,s_n E+\mbox{sh}\,s_n W\,.
$$

From Lemma (\ref{asymptotically_equivalence}) it follows that there exists the solution \(V_n\) of~(\ref{equation_V_n}) such that \(\lim\limits_{n\to\infty}V_n=\lim\limits_{n\to\infty}\tilde{V}_n\). To find out the structure of this solution one can use the integral equation~(\ref{asymptot_V_n}). One has
$$
V_n=e^{s_n W}\left(E-\frac{\alpha}{2}\int\limits_n^{+\infty}\frac{\sin(\beta t^s+\gamma)}{t}\,e^{-s_n W}\,(W_t-W)\,V_t\,dt\right)\,.
$$
Using this equation one can construct an infinite series for unknown solution \(V_n\). Substituting this formula into itself, we obtain
$$
V_n=e^{s_n W}\left(E-\frac{\alpha}{2}\int\limits_n^{+\infty}\frac{\sin(\beta t^s+\gamma)}{t}\,e^{-s_t W}\,(W_t-W)\,e^{s_{t} W}\,dt+\right.
$$
$$
\left.+\left(\frac{\alpha}{2}\right)^2\int\limits_n^{+\infty}\frac{\sin(\beta t^s+\gamma)}{t}\,e^{-s_{t} W}\,(W_t-W)\,e^{s_{t} W}\,dt\int\limits_t^{+\infty}\frac{\sin(\beta t_1^s+\gamma)}{t_1}\,e^{-s_{t_1} W}\,(W_{t_1}-W)\,e^{s_{t_1} W}\,dt_1\right)
$$
Continuing this process we obtain the following series
$$
V_n=e^{s_n W}\left(E+\int\limits_n^{+\infty}K_{t_1}\,dt_1+\int\limits_n^{+\infty}K_{t_1}\,dt_1\int\limits_{t_1}^{+\infty}K_{t_2}\,dt_2+\right.
$$
\begin{equation}
\label{series_for_V_n}
\left.+\int\limits_n^{+\infty}K_{t_1}\,dt_1\int\limits_{t_1}^{+\infty}K_{t_2}\,dt_2\int\limits_{t_2}^{+\infty}K_{t_3}\,dt_3+\ldots\right)\,,
\end{equation}
where
$$
K_t=-\frac{\alpha\sin(\beta t^s+\gamma)}{2t}\,e^{-s_t W}\,(W_t-W)\,e^{s_{t} W}\,.
$$
To prove the convergence of this series and for further consideration it is convenient to make the substitution \(y=\beta\,n^s\) and \(y_i=\beta\,t_i^s\) in each integral (we will regard \(\beta>0\) ). Besides 
$$
\lim\limits_{n\to\infty}V_n=\lim\limits_{n\to\infty}V(\beta\,n^s)=\lim\limits_{y\to+\infty}V(y)\,.
$$
One has
$$
V(y)=e^{s(y) W}\left(E+\int\limits_y^{+\infty}K(y_1)\,dy_1+\int\limits_y^{+\infty}K(y_1)\,dy_1\int\limits_{y_1}^{+\infty}K(y_2)\,dy_2+\right.
$$
$$
\left.+\int\limits_y^{+\infty}K(y_1)\,dy_1\int\limits_{y_1}^{+\infty}K(y_2)\,dy_2\int\limits_{y_2}^{+\infty}K(y_3)\,dy_3+\ldots\right)\,,\quad y\ge \beta\,,
$$
where
$$
V_n=V(\beta\,n^s)\,,\quad K(y)=-\frac{\alpha\sin(y+\gamma)}{2sy}\,e^{-s(y) W}\,(W(y)-W)\,e^{s(y) W}=
$$
$$
=-\frac{\alpha\sin(y+\gamma)}{2sy}\,W\,e^{-s(y) W}\,(WW(y)-E)\,e^{s(y) W}\,,\quad s(y)=\frac{\alpha}{2s}\int\limits_\beta^{y}\frac{\sin(t+\gamma)}{t}\,dt\,,	
$$
$$
WW(y)-E=\left(\begin{array}{cc} 
		e^{2(c(y)-c)}-1 & 0\\
        0 & e^{-2(c(y)-c)}-1
                \end{array}\right)\,,
$$
$$
c(y)=\frac{\alpha}{2s}\int\limits_\beta^{y}\frac{\cos(t+\gamma)}{t}\,dt\,,\quad c(y)-c=-\frac{\alpha}{2s}\int\limits_y^{+\infty}\frac{\cos(t+\gamma)}{t}\,dt\,.
$$
It follows that \(\|WW(y)-E\|\le\dfrac{{\tilde C}}{y}\) and hence \(\|K(y)\|\le\dfrac{C}{y^2}\). Therefore
$$
\left\|\int\limits_y^{+\infty}K(y_1)\,dy_1\int\limits_{y_1}^{+\infty}K(y_2)\,dy_2\ldots\int\limits_{y_{n-1}}^{+\infty}K(y_n)\,dy_n\right\|\le 
C^n\int\limits_y^{+\infty}\frac{dy_1}{y_1^2}\int\limits_{y_1}^{+\infty}\frac{dy_2}{y_2^2}\ldots\int\limits_{y_{n-1}}^{+\infty}\frac{dy_n}{y_n^2}=\frac{1}{n!}\left(\frac{C}{y}\right)^n\,.
$$
This estimate proves at once the convergence of series~(\ref{series_for_V_n}) and inequality
$$
\left\|e^{-s(y)W}V(y)\right\|\le e^{C/y}\,.
$$
We are interested in the dependence as \(x\to 0\). Remembering that \(\beta=-\dfrac{2x}{1-\alpha}\), 
\(\gamma=2xc\), we have
$$
c(y)\sim-\frac{\alpha}{2s}\ln|\beta|\,,\: x\to 0\,,\qquad c\sim-\frac{\alpha}{2s}\ln|\beta|\,,\: x\to 0\,,\quad e^c\sim\delta\,|x|^{-\dfrac{\alpha}{2(1-\alpha)}}\,,\:x\to 0\,,
$$
$$
\lim\limits_{x\to 0}s(y)=\frac{\alpha}{2s}\int\limits_0^{y}\frac{\sin t}{t}\,dt\,,\quad
\lim\limits_{x\to 0}(c(y)-c)=-\frac{\alpha}{2s}\int\limits_y^{+\infty}\frac{\cos t}{t}\,dt
$$
It follows that the most essential role at \(x\to 0\) belongs to the matrix \(W\). Taking into account that \(e^{\pm s(y)W}=\mbox{ch}\,s(y) E\pm\mbox{sh}\,s(y) W\) and \(W^2=E\), we can present the matrix \(K(y)\) in the form
$$
\left(\begin{array}{cc}
			a(y) & m(y)\, e^{-2c}\\
			n(y)\, e^{2c} & b(y)
\end{array}\right)\,.
$$
Because the product of such matrices has the same structure, the same form will have and the solution \(V(y)\). Let's try to find the matrix \(V(y)\) in this form. After substitution \(y=\beta\,n^s\) the equation~(\ref{equation_V_n}) will take the form
$$
\frac{d V(y)}{dy}=\frac{\alpha\sin(y+\gamma)}{2sy}\,W(y)\,V(y)
$$
or
$$
\frac{d V(y)}{dy}=s'(y)\,W(y)\,V(y)\,.
$$
We will look for \(V(y)\) such above
\begin{equation}
\label{form_of_V(y)}
V(y)=\left(\begin{array}{cc}
			a(y) & m(y)\, e^{-2c}\\
			n(y)\, e^{2c} & b(y)
\end{array}\right)\,,
\end{equation}
where the functions \(a(y), b(y), n(y), m(y)\) have a finite limit at \(y\to+\infty\) according to the Corollary~(\ref{finite_limit}).  After substitution in the equation we will have
$$
\left(\begin{array}{cc}
			a'(y) & m'(y)\, e^{-2c}\\
			n'(y)\, e^{2c} & b'(y)
\end{array}\right)=s'(y)\left(\begin{array}{cc}
			n(y)\,e^{2(c-c(y))} & b(y)\, e^{-2c(y)}\\
			a(y)\, e^{2c(y)} & m(y)\,e^{-2(c-c(y))}
\end{array}\right)
$$
and this matrix equation decays into two independent systems
\begin{equation}
\label{systems_of_equations}
\left\{\begin{array}{l}
		a'(y)=s'(y)\,n(y)\,e^{2(c-c(y))}\\
		n'(y)=s'(y)\,a(y)\,e^{-2(c-c(y))}
\end{array}\right.\quad 
\left\{\begin{array}{l}
		m'(y)=s'(y)\,b(y)\,e^{2(c-c(y))}\\
		b'(y)=s'(y)\,m(y)\,e^{-2(c-c(y))}
\end{array}\right.\,.
\end{equation}
Let us consider for example the first system. After differentiation of the first equation and using the second one we will have the following differential equation for \(a(y)\) :
$$
a''=\left(\ln[s'(y)\,e^{2(c-c(y))}]\right)'\,a'+(s'(y))^2\,a
$$
or
\begin{equation}
\label{differential_equation_a}
a''=\left(\frac{\cos(y+\gamma)}{\sin(y+\gamma)}-\frac{1}{y}-\frac{\alpha\cos(y+\gamma)}{sy}\right)a'+\frac{\alpha^2\sin^2(y+\gamma)}{4s^2y^2}\,a\,.
\end{equation}
The parameter \(x\) enters here trough the constant \(\gamma\) only. Supposing the next application for \(x\to 0\), we can regard \(0<|\gamma|<1\). Then this equation has \(y=0\) as a regular point. Governing equation for~(\ref{differential_equation_a}) has the form~\cite{10}
$$
p(p-1)+\left(1+\frac{\alpha\cos\gamma}{s}\right)p-\frac{\alpha^2\sin^2\gamma}{4s^2}=0
$$
or
$$
p^2+\frac{\alpha\cos\gamma}{s}\,p-\frac{\alpha^2\sin^2\gamma}{4s^2}=0\,,
$$
whence
$$
p_1=\frac{\alpha(-\cos\gamma+1)}{2s}>0\,,\quad p_2=\frac{\alpha(-\cos\gamma-1)}{2s}<0\,.
$$
Hence the equation~(\ref{differential_equation_a}) has two linearly independent solutions \(A_1(y)\) and \(A_2(y)\) which have the following behaviour at the point \(y=0\) :
\begin{equation}
\label{form_of_solutions_at_0}
A_1(y)=y^{p_1}\,(c_0+c_1y+\ldots)\,,\quad A_2(y)=y^{p_2}\,(c'_0+c'_1y+\ldots)+c_{-1}\,A_1(y)\,\ln y\,,\quad c_0,c'_0\ne0\,.
\end{equation}
The first solution \(A_1(y)\) is bounded as \(y\to+0\) but the second one is unbounded and \(A_2(y)\sim c'_0\,y^{p_2}\) as \(y\to+0\).

Let's find now the behaviour of solutions as \(y\to+\infty\). Notice that if we omit the last term with \(a(y)\) in~(\ref{differential_equation_a}), then the equation can be solved exactly. Namely, consider the equation
$$
a''=\left(\frac{\cos(y+\gamma)}{\sin(y+\gamma)}-\frac{1}{y}-\frac{\alpha\cos(y+\gamma)}{sy}\right)a'\,.
$$ 
By obvious transformation it reduces to the form
$$
\left(\ln a'\right)'=\left(\ln\left(\frac{\sin(y+\gamma)}{y}\,e^{2(c-c(y))}\right)\right)'
$$
whence 
$$
a'(y)=C\,\frac{\sin(y+\gamma)}{y}\,e^{2(c-c(y))}\,,\quad C=const
$$
and
$$
a(y)=C\int\frac{\sin(y+\gamma)}{y}\,e^{2(c-c(y))}\,dy\,.
$$
Let us consider now the differential equation~(\ref{differential_equation_a}) as inhomogeneous equation with inhomogeneous part
$$
f(y)=\frac{\alpha^2\sin^2(y+\gamma)}{4s^2y^2}\,a(y)\,.
$$
That is consider the equation
$$
a''-\left(\frac{\cos(y+\gamma)}{\sin(y+\gamma)}-\frac{1}{y}-\frac{\alpha\cos(y+\gamma)}{sy}\right)a'
=f(y)\,.
$$
Using two linearly independent solutions of homogeneous equation
$$
a_1(y)=1\,,\quad a_2(y)=\int\limits_1^y\frac{\sin(t+\gamma)}{t}\,e^{2(c-c(t))}\,dt\,,
$$
we can construct Green's function \(G(y,t)\) such that 
$$
a(y)=C_1+C_2\,a_2(y)+\tilde{a}(y)\,,\quad\tilde{a}(y)=\int\limits_0^{+\infty}G(y,t)f(t)\,dt\,,\quad \tilde{a}(+\infty)=\tilde{a}'(+\infty)=0\,.
$$ 
We have conditions
$$
G(y,t)=g_1\,a_2(y)+g_2\,,\quad y\ge t\,,
$$
$$
G(y,t)=g_3\,a_2(y)+g_4\,,\quad y\le t\,,
$$
$$
G(t+0,t)=G(t-0,t)\Rightarrow g_1\,a_2(t)+g_2=g_3\,a_2(t)+g_4\,,
$$
$$
G'_y(t+0,t)-G'_y(t-0,t)=1 \Rightarrow g_1\,a'_2(t)-g_3\,a'_2(t)=1\,,
$$
$$
G(+\infty,t)=0 \Rightarrow g_1\,a_2(+\infty)+g_2=0\,,
$$
$$
G'_y(+\infty,t)=0 \Rightarrow g_1\,a'_2(+\infty)=0\,.
$$
Solving this system, we obtain
$$
G(y,t)=0\,,\quad y\ge t\,,
$$
$$
G(y,t)=\frac{a_2(t)-a_2(y)}{a'_2(t)}\,,\quad y\le t\,.
$$
For \(\tilde{a}(y)\) we will have then
$$
\tilde{a}(y)=\int\limits_y^{+\infty}\frac{a_2(t)-a_2(y)}{a'_2(t)}\,f(t)\,dt\,.
$$
Since
$$
\frac{f(t)}{a'_2(t)}=\frac{\alpha^2\sin(t+\gamma)}{4s^2t}\,a(t)\,e^{-2(c-c(t))}\,,
$$
all conditions are fulfilled and integral is convergent (remember that \(a(t)\to a\,,\:|a|<\infty\) ). Substituting the expression for \(\dfrac{f(t)}{a'_2(t)}\), we obtain
\begin{equation}
\label{integral_equation_for_a}
a(y)=C_1+C_2\,a_2(y)+\frac{\alpha^2}{4s^2}\int\limits_y^{+\infty}\frac{\sin(t+\gamma)}{t}\,(a_2(t)-a_2(y))\,a(t)\,e^{-2(c-c(t))}\,dt\,.
\end{equation}

Let's make some remarks about convergence of integrals. First, there exists the limit of \(a_2(y)\) at \(y\to+\infty\). Actually, one has
$$
a_2(+\infty)=\int\limits_1^{+\infty}\frac{\sin(t+\gamma)}{t}\,(e^{2(c-c(t))}-1)\,dt+\int\limits_1^{+\infty}\frac{\sin(t+\gamma)}{t}\,dt\,.
$$
Second integral is conditionally convergent by Dirichlet's test. The first integral is absolutely convergent because
$$
(e^{2(c-c(t))}-1)\approx \frac{1}{t}\,,\quad t\to+\infty\,.
$$
Besides, one has
$$
a_2(+\infty)-a(y)=O\left(\frac{1}{y}\right)\,,\quad y\to+\infty\,.
$$

Putting in (\ref{integral_equation_for_a}) first \(C_1=1\), \(C_2=0\) and then \(C_1=0\), \(C_2=1\) we obtain two integral equations
\begin{equation}
\label{integral_equation_for_A1}
a(y)=1+\frac{\alpha^2}{4s^2}\int\limits_y^{+\infty}\frac{\sin(t+\gamma)}{t}\,(a_2(t)-a_2(y))\,a(t)\,e^{-2(c-c(t))}\,dt\,,
\end{equation}
and
\begin{equation}
\label{integral_equation_for_A2}
a(y)=a_2(y)+\frac{\alpha^2}{4s^2}\int\limits_y^{+\infty}\frac{\sin(t+\gamma)}{t}\,(a_2(t)-a_2(y))\,a(t)\,e^{-2(c-c(t))}\,dt\,.
\end{equation}

Denote by \(A_1(y)\) the solution of~(\ref{integral_equation_for_A1}) and by \(A_2(y)\) the solution of~(\ref{integral_equation_for_A2}). \(A_1(y)\) and \(A_2(y)\) are also the solutions of differential equation~(\ref{differential_equation_a}). They are linearly independent. Actually, one can show that the solution \(A_1(y)\) is bounded at \(y\to+0\) but the solution \(A_2(y)\) is not. One has
$$
a_2(y)=\int\limits_1^{y}\frac{\sin(t+\gamma)}{t}\,e^{2(c-c(t))}\,dt\,,\quad
c-c(t)=\frac{\alpha}{2s}\int\limits_t^{+\infty}\frac{\cos(z+\gamma)}{z}\,dz\,.
$$
The expression \(|c-c(t)|\) decreases as \(\dfrac{1}{t}\) for large \(t\), but as \(t\to+0\) we have
$$
c-c(t)=\frac{\alpha}{2s}\int\limits_t^{+\infty}\frac{\cos(z+\gamma)}{z}\,dz\sim-\frac{\alpha\cos\gamma}{2s}\,\ln t\,,\quad e^{2(c-c(t))}\approx t^{\textstyle-\frac{\alpha\cos\gamma}{s}}\,.
$$
Then
$$
|a_2(y)|\approx y^{\textstyle-\frac{\alpha\cos\gamma}{s}}\,,\quad y\to+0\,.
$$

So the solution \(A_1(y)\) of~(\ref{integral_equation_for_A1}) satisfies the first decomposition in~(\ref{form_of_solutions_at_0}) at \(y\to+0\) and is bounded. The unboundedness of \(A_2(y)\) at \(y\to+0\) follows from integral equation
$$
A_2(y)=a_2(y)+\frac{\alpha^2}{4s^2}\int\limits_y^{+\infty}\frac{\sin(t+\gamma)}{t}\,(a_2(t)-a_2(y))\,A_2(t)\,e^{-2(c-c(t))}\,dt\,.
$$
Actually, if we suppose that \(\sup\limits_{y>0}|A_2(y)|<+\infty\) then the integral on the right-hand side of this equation is bounded as \(y\to+0\). Hence \(a_2(y)\) should be bounded at \(y\to+0\). But it is not true. 

Thus we have proved

\begin{theorem}
\label{properties_of_solution_of_diff_equation}
The differential equation~(\ref{differential_equation_a}) has two linearly independent solutions \(A_1(y)\) and \(A_2(y)\) such that
$$
\lim\limits_{y\to+\infty}A_1(y)=1\,,\quad \lim\limits_{y\to+\infty}A_2(y)=\int\limits_1^{+\infty}\frac{\sin(t+\gamma)}{t}\,e^{2(c-c(t))}\,dt\,,
$$
$$
A_1(y)\sim c_0\,y^{p_1}\,,\quad A_2(y)\sim c'_0\,y^{p_2}\,,\quad y\to+0\,,\quad(c_0,c'_0\ne0)\,,
$$
where
$$
p_1=\frac{\alpha(-\cos\gamma+1)}{2s}>0\,,\quad p_2=\frac{\alpha(-\cos\gamma-1)}{2s}<0\,.
$$
The constants \(c_0,c'_0\) have the following behaviour at \(\gamma\to0\,:\)
$$
\lim\limits_{\gamma\to0}c_0=\tilde{c}_0\ne0\,,\quad c'_0\approx\sin\gamma\,.
$$
\end{theorem}
\begin{proof}
It remains to prove the last limit properties of constants \(c_0,c'_0\). Consider the equation~(\ref{differential_equation_a})
$$
a''=\left(\frac{\cos(y+\gamma)}{\sin(y+\gamma)}-\frac{1}{y}-\frac{\alpha\cos(y+\gamma)}{sy}\right)a'+\frac{\alpha^2\sin^2(y+\gamma)}{4s^2y^2}\,a\,.
$$
If \(\gamma\to0\) it takes the form
\begin{equation}
\label{differential_equation_a_gamma=0}
a''=\left(\frac{\cos y}{\sin y}-\frac{1}{y}-\frac{\alpha\cos y}{sy}\right)a'+\frac{\alpha^2\sin^2 y}{4s^2y^2}\,a\,.
\end{equation}
Governing equation at the regular point \(y=0\) is
$$
p(p-1)+\frac{\alpha}{s}\,p=0\,,
$$
whence
$$
\tilde{p}_1=0\,,\quad \tilde{p}_2=1-\frac{\alpha}{s}=\frac{1-2\alpha}{1-\alpha}<0\,.
$$
We have thus
\begin{equation}
\label{limits_of_p1_and_p2}
\lim\limits_{\gamma\to0}p_1=\tilde{p}_1\,,\qquad 1+\lim\limits_{\gamma\to0}p_1=\tilde{p}_2\,.
\end{equation}
Two linearly independent solutions of~(\ref{differential_equation_a_gamma=0}) at \(y\to 0\) have
the form
$$
\tilde{A}_1(y)=\tilde{c}_0+\tilde{c}_1y+\ldots\,,\quad \tilde{A}_2(y)=y^{\tilde{p}_2}\,(\tilde{c}'_0+\tilde{c}'_1y+\ldots)+\tilde{c}_{-1}\,\tilde{A}_1(y)\,\ln y\,,\quad \tilde{c}_0,\tilde{c}'_0\ne0\,.
$$
Taking into account~(\ref{limits_of_p1_and_p2}) one can conclude that to obtain this formulas for \(\tilde{A}_1(y)\) and \(\tilde{A}_2(y)\) from~(\ref{form_of_solutions_at_0}) we should have
$$
\lim\limits_{\gamma\to0}c_0=\tilde{c}_0\ne0\,,\quad \lim\limits_{\gamma\to0}c'_0=0\,.
$$
From~(\ref{differential_equation_a}) we have for Wronskian W(y) the expression
$$
W(y)=A_1(y)A'_2(y)-A'_1(y)A_2(y)=W_0\frac{\sin(y+\gamma)}{y}\,e^{2(c-c(y))}\,.
$$
At \(y\to0\) it follows that
$$
c_0\,c'_0(p_2-p_1)\,y^{p_1+p_2-1}=W_0\,\sin\gamma\,y^{{\textstyle-\frac{\alpha\cos\gamma}{s}-}1}\,,
$$ 
or
$$
c_0\,c'_0(-\frac{\alpha}{s})\,y^{{\textstyle-\frac{\alpha\cos\gamma}{s}-}1}=W_0\,\sin\gamma\,y^{{\textstyle-\frac{\alpha\cos\gamma}{s}-}1}\,.
$$ 
As \(c_0\to\tilde{c}_0\ne0\) it follows that \(c'_0\approx\sin\gamma\), \(\gamma\to0\).
\end{proof}

Now we are able to state the dependence of limit value \(a(+\infty)\) from initial conditions. The general solution of~(\ref{differential_equation_a}) has the form
$$
a(y)=C_1\,A_1(y)+C_2A_2(y)\,.
$$
For determining of the constants \(C_1,C_2\) we have the system
$$
\left\{\begin{array}{l}
C_1\,A_1(\beta)+C_2\,A_2(\beta)=a(\beta)\\
C_1\,A'_1(\beta)+C_2\,A'_2(\beta)=a'(\beta)
\end{array}\right.\,,
$$
whence
$$
C_1=\frac{a(\beta)\,A'_2(\beta)-a'(\beta)\,A_2(\beta)}{A_1(\beta)\,A'_2(\beta)-A_2(\beta)\,A'_1(\beta)}\,,\quad
C_2=\frac{a'(\beta)\,A_1(\beta)-a(\beta)\,A'_1(\beta)}{A_1(\beta)\,A'_2(\beta)-A_2(\beta)\,A'_1(\beta)}\,.
$$
From Theorem~(\ref{properties_of_solution_of_diff_equation}) it follows that
$$
a(+\infty)=C_1+C_2\int\limits_1^{+\infty}\frac{\sin(t+\gamma)}{t}\,e^{2(c-c(t))}\,dt\,.
$$

Let's consider now the limit of \(a(+\infty)\) when \(x\to0\). As the density is even function we can consider the limit \(x\to-0\). Then \(\beta\to+0\), \(|\gamma|\to+0\) ( \(\beta=-\dfrac{2x}{1-\alpha}\), \(\gamma=2xc\) ). If
$$
V_1=\left(\begin{array}{cc}
V_{11} & V_{12}\\
V_{21} & V_{22}
\end{array}
\right)
$$
then from~(\ref{form_of_V(y)}), (\ref{systems_of_equations}) it follows that
$$
V_{11}=a(\beta)\,,\quad V_{21}=n(\beta)\,e^{2c}\,,\quad a'(\beta)=s'(\beta)\,V_{21}=\frac{\alpha\sin(\beta+\gamma)}{2s\beta}\,V_{21}\,.
$$
Using the results of the Theorem~(\ref{properties_of_solution_of_diff_equation}) and taking into account~(\ref{V1_xto0}) we have
$$
\lim\limits_{x\to-0}C_1=\lim\limits_{x\to-0}\frac{V_{11}\,c'_0\,p_2\,\beta^{\,p_2-1}-c'_0\,\beta^{\,p_2}\,s'(\beta)\,V_{21}}{c_0\,c'_0\,(p_2-p_1)\,\beta^{\,p_1+p_2-1}}=\lim\limits_{x\to-0}\frac{V_{11}\,p_2}{c_0\,(p_2-p_1)\,\beta^{\,p_1}}=\frac{v_{11}}{\tilde{c}_0}\,,
$$
because \(\dfrac{p_2}{p_2-p_1}=\dfrac{1+\cos\gamma}{2}\to 1\), \(\beta^{\,p_1}\to 1\)\,.
One has also as \(x\to-0\)
\begin{equation}
\label{limit_C2}
C_2\sim
\frac{s'(\beta)\,V_{21}\,c_0\,\beta^{\,p_1}-V_{11}\,c_0\,p_1\,\beta^{\,p_1-1}}{c_0\,c'_0\,(p_2-p_1)\,\beta^{\,p_1+p_2-1}}
=\frac{s'(\beta)\,V_{21}\,\beta-V_{11}\,p_1}{c'_0\,(-\frac{\alpha}{s})\,\beta^{\,p_2}}=O\left(|x|^{1+\alpha/s}\right)\to0\,,
\end{equation}
$$
\int\limits_1^{+\infty}\frac{\sin(t+\gamma)}{t}\,e^{2(c-c(t))}\,dt\to\int\limits_1^{+\infty}\frac{\sin t}{t}\,\exp\left\{\frac{\alpha}{s}\int\limits_t^{+\infty}\frac{cos z\,dz}{z}\right\}\,dt=I=const.
$$
Thus we obtain
\begin{equation}
\label{limit_a(infty)}
\lim\limits_{x\to-0}a(+\infty)=\frac{v_{11}}{\tilde{c}_0}\ne0\,.
\end{equation}

The limit of the function \(n(y)\) at \(y\to+\infty\) we will find from the first equation of the system~(\ref{systems_of_equations})
$$
n(+\infty)=\lim\limits_{y\to+\infty}\frac{a'(y)}{s'(y)}\,e^{-2(c-c(y))}=\lim\limits_{y\to+\infty}\frac{a'(y)}{s'(y)}=\lim\limits_{y\to+\infty}a'(y)\,\frac{2sy}{\alpha\sin(y+\gamma)}\,.
$$
One has
$$
a'(y)=C_1\,A'_1(y)+C_2\,A'_2(y)\,.
$$
From integral equations
$$
A_1(y)=1+\frac{\alpha^2}{4s^2}\int\limits_y^{+\infty}\frac{\sin(t+\gamma)}{t}\,(a_2(t)-a_2(y))\,A_1(t)\,e^{-2(c-c(t))}\,dt\,,
$$
$$
A_2(y)=a_2(y)+\frac{\alpha^2}{4s^2}\int\limits_y^{+\infty}\frac{\sin(t+\gamma)}{t}\,(a_2(t)-a_2(y))\,A_2(t)\,e^{-2(c-c(t))}\,dt
$$
it follows that
$$
A'_1(y)=\frac{\alpha^2}{4s^2}\int\limits_y^{+\infty}\frac{\sin(t+\gamma)}{t}\,(a_2(t)-a'_2(y))\,A_1(t)\,e^{-2(c-c(t))}\,dt=
$$
$$
=\frac{\alpha^2}{4s^2}\int\limits_y^{+\infty}\frac{\sin(t+\gamma)}{t}\,a_2(t)\,A_1(t)\,e^{-2(c-c(t))}\,dt-
$$
$$
-\frac{\sin(y+\gamma)}{y}\,e^{2(c-c(y))}\frac{\alpha^2}{4s^2}\int\limits_y^{+\infty}\frac{\sin(t+\gamma)}{t}\,A_1(t)\,e^{-2(c-c(t))}\,dt=
$$
$$
=\frac{\alpha^2\sin(y+\gamma)}{4s^2y}\,a_2(y)\,A_1(y)\,e^{-2(c-c(y))}+O\left(\frac{\sin(y+\gamma)}{y^2}\right)\,,\quad y\to+\infty\,,
$$
$$
A'_2(y)=\frac{\sin(y+\gamma)}{y}\,e^{2(c-c(y))}+\frac{\alpha^2}{4s^2}\int\limits_y^{+\infty}\frac{\sin(t+\gamma)}{t}\,(a_2(t)-a'_2(y))\,A_2(t)\,e^{-2(c-c(t))}\,dt=
$$
$$
=\frac{\sin(y+\gamma)}{y}\,e^{2(c-c(y))}+\frac{\alpha^2}{4s^2}\int\limits_y^{+\infty}\frac{\sin(t+\gamma)}{t}\,a_2(t)\,A_2(t)\,e^{-2(c-c(t))}\,dt-
$$
$$
-\frac{\sin(y+\gamma)}{y}\,e^{2(c-c(y))}\frac{\alpha^2}{4s^2}\int\limits_y^{+\infty}\frac{\sin(t+\gamma)}{t}\,A_2(t)\,e^{-2(c-c(t))}\,dt=
$$
$$
=\frac{\sin(y+\gamma)}{y}\,e^{-2(c-c(y))}\left(1+\frac{\alpha^2\,a_2(y)\,A_2(y))}{4s^2}\right)+O\left(\frac{\sin(y+\gamma)}{y^2}\right)\,,\quad y\to+\infty\,.
$$
Here we used integrating by parts and boundedness of function of the form
$$
\int\limits_1^y\sin t\,f(t)\,dt\,,\quad \lim\limits_{y\to+\infty}f(t)=f_0<\infty\,,\quad
f'(t)\approx\frac{\sin t}{t}\,.
$$
Thus we obtain
$$
n(+\infty)=C_1\lim\limits_{y\to+\infty}A_1'(y)\,\frac{2sy}{\alpha\sin(y+\gamma)}+
C_2\lim\limits_{y\to+\infty}A_2'(y)\,\frac{2sy}{\alpha\sin(y+\gamma)}=
$$
$$
=C_1\,\frac{\alpha}{2s}\,a_2(+\infty)\,A_1(+\infty)+C_2\,\left(\frac{2s}{\alpha}+\frac{\alpha}{2s}\,a_2(+\infty)\,A_2(+\infty)\right)=
$$
$$
=\frac{\alpha}{2s}\,a_2(+\infty)\,a(+\infty)+C_2\,\frac{2s}{\alpha}\,.
$$
Using~(\ref{limit_C2}) and~(\ref{limit_a(infty)}) we find
\begin{equation}
\label{limit_n(infty)}
\lim\limits_{x\to-0}n(+\infty)=\frac{\alpha\,v_{11}}{2s\,\tilde{c}_0}\,I\ne0\,.
\end{equation}

As the second system in~(\ref{systems_of_equations}) has the same form as first system we can apply all results for \(a(y)\) and \(n(y)\) to the functions \(m(y)\) and \(b(y)\). One has
$$
m(y)=C_1\,A_1(y)+C_2A_2(y)\,,
$$
$$
C_1=\frac{m(\beta)\,A'_2(\beta)-m'(\beta)\,A_2(\beta)}{A_1(\beta)\,A'_2(\beta)-A_2(\beta)\,A'_1(\beta)}\,,\quad
C_2=\frac{m'(\beta)\,A_1(\beta)-m(\beta)\,A'_1(\beta)}{A_1(\beta)\,A'_2(\beta)-A_2(\beta)\,A'_1(\beta)}\,,
$$
$$
m(+\infty)=C_1+C_2\int\limits_1^{+\infty}\frac{\sin(t+\gamma)}{t}\,e^{2(c-c(t))}\,dt\,.
$$
Taking into account that
$$
m(\beta)=V_{12}\,e^{2c}\,,\quad m'(\beta)=s'(\beta)\,V_{22}\,e^{2c}\,,
$$
we have at \(x\to-0\)
$$
e^{-2c}\,C_1\to\frac{V_{12}\,c'_0\,p_2\,\beta^{\,p_2-1}-c'_0\,\beta^{\,p_2}\,s'(\beta)\,V_{22}}{c_0\,c'_0\,(p_2-p_1)\,\beta^{\,p_1+p_2-1}}=\frac{V_{12}\,p_2-\beta\,s'(\beta)\,V_{22}}{c_0\,(p_2-p_1)\,\beta^{\,p_1}}=O(x)\,,
$$
$$
e^{-2c}\,C_2\sim
\frac{s'(\beta)\,V_{22}\,c_0\,\beta^{\,p_1}-V_{12}\,c_0\,p_1\,\beta^{\,p_1-1}}{c_0\,c'_0\,(p_2-p_1)\,\beta^{\,p_1+p_2-1}}
=\frac{s'(\beta)\,V_{22}\,\beta-V_{12}\,p_1}{c'_0\,(-\frac{\alpha}{s})\,\beta^{\,p_2}}=O\left(|x|^{1+\alpha/s}\right)\to0\,,
$$
so that
\begin{equation}
\label{limit_m(infty)}
m(+\infty)\,e^{-2c}=O(x)\,.
\end{equation}
For the function \(b(y)\) we obtain
$$
b(+\infty)=\lim\limits_{y\to+\infty}\frac{m'(y)}{s'(y)}\,e^{-2(c-c(y))}=\lim\limits_{y\to+\infty}\frac{m'(y)}{s'(y)}=\lim\limits_{y\to+\infty}m'(y)\,\frac{2sy}{\alpha\sin(y+\gamma)}=
$$
\begin{equation}
\label{limit_b(infty)}
=\frac{\alpha}{2s}\,a_2(+\infty)\,m(+\infty)+C_2\,\frac{2s}{\alpha}=O\left(|x|^{1-\textstyle\frac{\alpha}{s}}\right)\,.
\end{equation}
From (\ref{limit_a(infty)}), (\ref{limit_n(infty)}), (\ref{limit_m(infty)}), (\ref{limit_b(infty)}) it follows
\begin{theorem}
\label{limit_V}
If \(V=\lim\limits_{n\to\infty}V_n\), then 
$$
V=\left(\begin{array}{cc}
\displaystyle\frac{v_{11}}{\tilde{c}_0} & O(x)\\
\displaystyle\frac{\delta\,\alpha\,v_{11}}{2s\,\tilde{c}_0}\,I\,|x|^{\textstyle-\frac{\alpha}{s}} & O\left(|x|^{1-\textstyle\frac{\alpha}{s}}\right)
\end{array}\right)\,,\quad x\to0\,,\quad \frac{\delta\,v_{11}}{\tilde{c}_0}\,I\ne0\,.
$$
\end{theorem}
\begin{corollary}
The matrix \(G=Z\,V\) have the following behaviour at \(x\to0\)
$$
G=\left(\begin{array}{cc}
\displaystyle\frac{v_{11}}{\tilde{c}_0}\,|x|^{\textstyle-\frac{\alpha}{2s}} & O(|x|^{1-\textstyle\frac{\alpha}{2s}})\\
\displaystyle\frac{\delta\,\alpha\,v_{11}}{2s\,\tilde{c}_0}\,I\,|x|^{\textstyle-\frac{\alpha}{2s}} & O\left(|x|^{1-\textstyle\frac{\alpha}{2s}}\right)
\end{array}\right)\,,\quad s=1-\alpha\,.
$$
\end{corollary}
This result instead of (\ref{G_xto0}) also leads to the main Theorem~(\ref{main_theorem}).

\section{Numerical illustrations}

In this section we illustrate the phenomenon of 2nd order spectral phase transition by numerical calculations. On the Fig.1 it is shown the spectral density \(f(x)\) for \(\alpha=0.6<2/3\). We can see the sharp peak near the point \(x=0\). On the Fig.2 it is shown the spectral density for \(\alpha=2/3\). In this case \(f(0)\) is finite. And the Fig.3 shows the spectral density for \(\alpha=0.8>2/3\). We can see that \(f(x)\) vanishes as \(x\to 0\).

\begin{figure}[h]
\begin{center}
\begin{minipage}[h]{0.29\linewidth}
\includegraphics[width=1.3\linewidth]{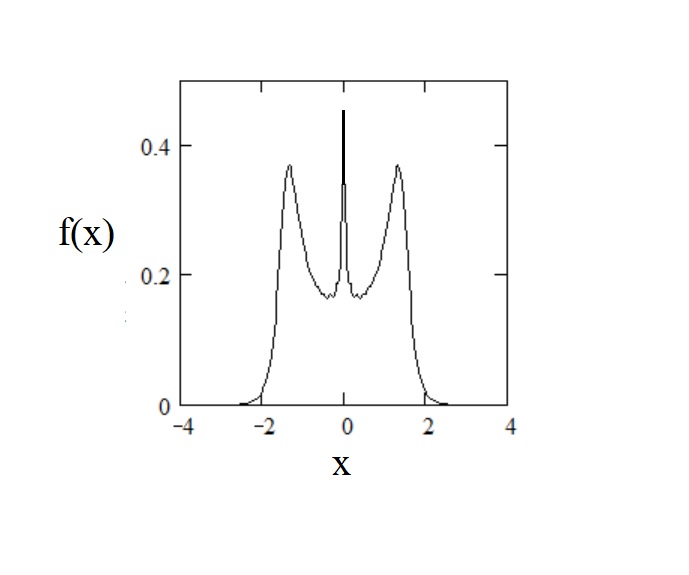}
\caption{$\alpha=0.6$}
\label{setminus}
\end{minipage}
\hfill 
\begin{minipage}[h]{0.29\linewidth}
\includegraphics[width=1.3\linewidth]{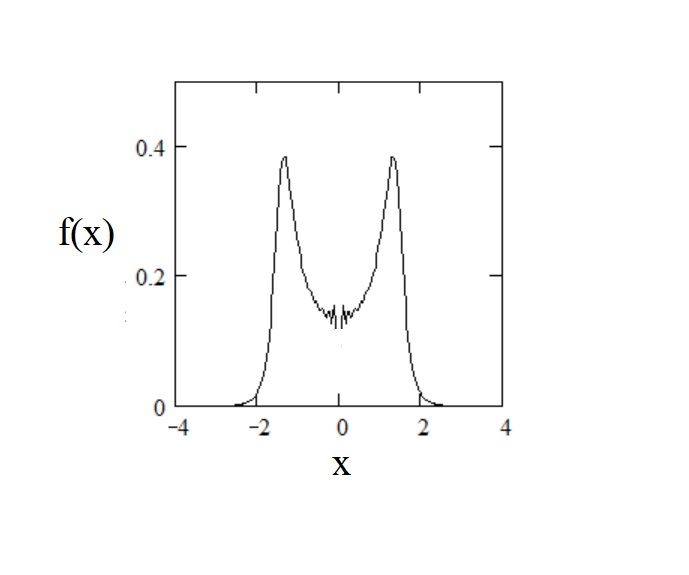}
\caption{$\alpha=2/3$}
\label{symmetrysetminus}
\end{minipage}
\hfill
\begin{minipage}[h]{0.29\linewidth}
\includegraphics[width=1.3\linewidth]{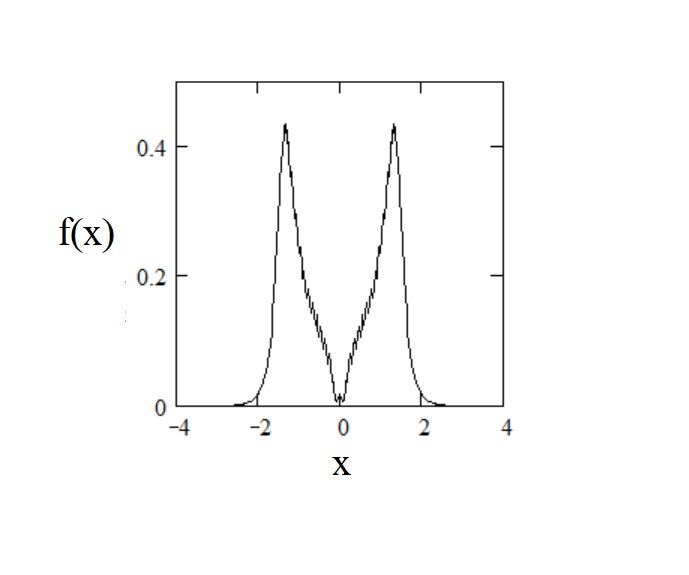}
\caption{$\alpha=0.8$}
\label{dopset}
\end{minipage}
\end{center}
\end{figure}

\section{Conclusion}

We have proved our main result ( Theorem~(\ref{main_theorem}) ) by changing discrete system~(\ref{discrete_system}) by continuous analogue - differential equation~(\ref{equation_G_n}). In the next work we will show how to obtain the same results in a discrete case with system~(\ref{discrete_system}).

\vspace{0.5cm}

{\bf Acknowledgment. }  I am grateful to Sergey Simonov for the discussion of the results and useful remarks.

\end{document}